\documentclass[11pt]{article}
\usepackage{geometry}                
\usepackage{graphicx}
\usepackage{amssymb}
\usepackage{epstopdf}
\usepackage{amsmath}
\usepackage{amscd}
\usepackage{amsthm}
\usepackage{mathrsfs}
\usepackage{amscd}
\usepackage{color}
\usepackage[T1]{fontenc}
\usepackage{CJKutf8}
\usepackage[english]{babel}
\usepackage{tikz,tikz-3dplot}
\usepackage{bm}
\usepackage{xcolor}
\usepackage[normalem]{ulem}
\usepackage[numbers,sort&compress]{natbib}
\usepackage{authblk}

\newcommand{\pow}{\mathcal{P}}
\newcommand{\pownon}{\mathcal{P}_0}

\newtheorem{thm}{Theorem}
\newtheorem{prop}[thm]{Proposition}
\newtheorem{cor}[thm]{Corollary}

\newtheorem{lem}[thm]{Lemma}

\newtheorem{dfn}{Definition}

\newcommand{\sA}{\mathcal{A}}

\newcommand{\sE}{\mathcal{E}}

\newcommand{\sU}{\mathcal{U}}
\newcommand{\sJ}{\mathcal{J}}

\newcommand{\bx}{\mathbf{x}}
\newcommand{\by}{\mathbf{y}}

\newcommand{\bc}{\mathbf{c}}
\newcommand{\bd}{\mathbf{d}}
\newcommand{\ba}{\mathbf{a}}

\newcommand{\Int}{\mathrm{Int}}
\newcommand{\Ext}{\mathrm{Ext}}

\newcommand{\bone}{\mathbf{1}}
\newcommand{\dn}{\delta_{\mathrm{neg}}}
\renewcommand{\dh}{\delta_H}
\newcommand{\sY}{\mathcal{Y}}

\newcommand{\sX}{\mathcal{X}}

\renewcommand{\Re}{\mathbb{R}}

\title{Negative type diversities, a multi-dimensional analogue of negative type metrics}

\author[$\dagger$]{Pei Wu}
\author[$\star$]{David Bryant}
\author[$\ddag$]{Paul Tupper}

\affil[$\dagger$]{Department of Mathematics, Simon Fraser University, Burnaby, Canada. Email~\texttt{wpei@sfu.ca}}
\affil[$\star$]{Department of Mathematics and Statistics, University of Otago, Dunedin, New Zealand. Email~\texttt{david.bryant@otago.ac.nz}}
\affil[$\ddag$]{Department of Mathematics, Simon Fraser University, Burnaby, Canada. Email~\texttt{pft@sfu.ca}}

\begin{document}

\maketitle

\begin{abstract}
Diversities are a generalization of metric spaces in which a non-negative value is assigned to all finite subsets of a set, rather than just to pairs of points. Here we provide an analogue of the theory of negative type metrics for diversities. We introduce negative type diversities, and show that, as in the metric space case, they are a generalization of $L_1$-embeddable diversities. 
We provide a number of characterizations of negative type diversities, including a geometric characterisation.  Much of the recent interest in negative type metrics stems from the connections between metric embeddings and approximation algorithms. We extend some of this work into the diversity setting, showing that lower bounds for embeddings of negative type metrics into $L_1$ can be extended to diversities by using recently established extremal results on hypergraphs. 
\end{abstract}

\section{Introduction} \label{sec:Intro}

A finite metric $(X,d)$ has {\em negative type} if for all  vectors $\bx$ indexed by the elements of $X$ with zero sum we have
\begin{equation}
\sum_{a \in X} \sum_{b \in X} \bx_a \bx_b d(a,b) \leq 0. \label{def:negMetric}
\end{equation}
Negative type metrics were introduced by Schoenberg \cite{schoenberg1935remarks,schoenberg1938metric}, who showed that
a finite metric has negative type exactly when the square of that metric is embeddable in Euclidean space (see, e.g. \cite{Deza1997}).  The concept can be extended to metrics with $p$-negative type, which satisfy 
\begin{equation}
\sum_{a \in X} \sum_{b \in X} \bx_a \bx_b d(a,b)^p \leq 0
\end{equation}
for all $\bx$ with zero sum (e.g. \cite{Sanchez15}).

Negative type metrics have received recent attention as a tool for combinatorial optimization based on metric embedding. Arora et al.\ \cite{arora2008euclidean} showed that any metric of negative type on an $n$-point set  can be embedded in $L_1$ with at most $O(\sqrt{\log n} \log \log n)$ distortion, an upper bound that closely matches the provable lower bound of $\Omega(\sqrt{\log n})$. Their results provide near-optimal approximation algorithms for {\sc{Sparsest Cut}}  and other key graph  problems.

This application of negative type metrics to combinatorial optimization continues a large body of work tracing back to the influential paper of Linial et al.~\cite{linial1995geometry}. A finite metric $(X,d)$ can be embedded in $L_1$ with distortion $c$ if there exists an isometrically $L_1$-embeddable $(X,d_1)$ such that 
\[d(x,y) \leq d_1(x,y) \leq c d(x,y)\]
for all $x,y$. Bourgain~\cite{bourgain1985lipschitz} showed that every finite metric on $n$ points can be embedded in $L_1$ with distortion $O(\log n)$. Linial et al.~\cite{linial1995geometry} showed how this provides an $O(\log n)$ approximation algorithm for {\sc Sparsest Cut}. 

In recent work, results on metric embedding and $L_1$-embeddable metrics have been generalized to {\em diversities} \cite{BryantTupper2014}. A  diversity $(X,\delta)$ is like a metric that assigns values to finite subsets of points, rather than just pairs. More formally,  given a set $X$ and a map $\delta$ from the finite subsets of $X$ to $\Re$, a diversity satisfies
\begin{quotation}
\noindent (D1) $\delta(A) \geq 0$, and $\delta(A) = 0$ if 
and only if 
$|A|\leq 1$;\\
(D2) If $B \neq \emptyset$ then $\delta(A\cup B) + \delta(B \cup C) \geq \delta(A \cup C).$
\end{quotation}
The first property corresponds to the metric axiom that $d(x,y) \geq 0$ for all $x,y$, and $d(x,y) = 0 \Leftrightarrow x=y$. The second property corresponds to the triangle inequality. Indeed if $(X,\delta)$ is a diversity and $d(x,y) = \delta(\{x,y\})$ for all $x,y$ then (D1) and (D2) imply that $(X,d)$ is a metric. Every diversity induces a metric in this fashion. The mathematics of diversities has been explored in  \cite{Bryant2012,BryantTupper2014,Bryant2016,BryantTupper2017open,bryant2018,Espinola14,kirk2014diversities,Piatek2014a,poelstra2013topological}. Much of this work parallels developments in metric theory. 

Many of the important results on $L_1$ embeddings and their applications have analogues for diversities. There is a natural diversity analogue of $L_1$-embeddable metrics:  A finite diversity $(X,\delta_1)$ is $L_1$-embeddable if there is a map $\phi:X \rightarrow \Re^m$ for some $m$ such that for all $A \subseteq X$,
\[\delta_1(A) = \sum_{i=1}^m \max\{\phi(a)_i - \phi(b)_i:a,b \in A\}.\]
Note that the induced metric of an  $L_1$-embeddable diversity is an $L_1$-embeddable metric. A diversity $(X,\delta)$ is said to be $L_1$-embeddable {\em with distortion $\alpha$} if there is an $L_1$-embeddable diversity $(X,d_1)$ such that 
\[\delta_1(A) \leq \delta(A) \leq \alpha \delta_1(A).\]
The main result of \cite{BryantTupper2014} is that links between $L_1$ embedding of metrics and {\sc{Sparsest Cut}} generalise  to a link between $L_1$ embeddings of diversities and {\sc{Hypergraph Sparsest Cut}}.  A bound of $\alpha(n)$ on the distortion required to embed an $n$-point diversity gives an $\alpha(n)$ approximation for {\sc{Hypergraph Sparsest Cut}}.  Unfortunately, we have been unable to prove a distortion bound better that $O(n)$, although much tighter bounds hold for  specific instances \cite{BryantTupper2017open}.

Given that the tightest approximation bounds for {\sc{Sparsest Cut}} are achieved by methods based on negative type metrics, an obvious question is whether equivalent results might be obtained for diversities. The first step is to determine what the appropriate definition of a negative type diversity might be. There are two main characterisations for negative type metrics: the negativity condition in \eqref{def:negMetric} and the fact that the squares of negative type metrics are Euclidean. We do not, yet, have an sufficiently convincing analogue of Euclidean diversities, though \eqref{def:negMetric} appears to generalize quite naturally. 

\begin{dfn} \label{def:negDiv} A finite diversity $(X,\delta)$ is of {\em negative type} if for all zero sum vectors $\bx \in \Re^{\pow(X)}$ with $\bx_\emptyset = 0
$ we have 
\[\sum_{A} \sum_{B} \bx_A \bx_B \delta(A \cup B) \leq 0.\]
\end{dfn}

In Section~\ref{sec:char} we present multiple characterisations of negative type diversities. Our first main theorem (Theorem~\ref{thm:negLambda}) gives an explicit characterisation of negative type diversities in terms of a finite set of linear inequalities, thereby demonstrating that the collection of negative type diversities on a set forms a polyhedral cone. This property does not hold for negative type metrics. We use the result to show that the space of negative type diversities spans the space of all diversities on that set. In contrast, the space of $L_1$-embeddable diversities, which is contained in the space of negative type diversities, has dimension roughly half of that of diversities. 

The induced metric of any $L_1$-embeddable diversity is an $L_1$-embeddable metric and, conversely, every $L_1$-embeddable metric is the induced metric of some $L_1$-embeddable diversity \cite{BryantTupper2014}. Turning to negative type diversities, we see that the induced metric of any negative type diversity is not just a negative type metric, it is also $L_1$-embeddable (Theorem~\ref{thm:inducedL1}). This means that negative type metrics which are not $L_1$-embeddable are not the induced metrics of any negative type diversity. We prove Theorem~\ref{thm:inducedL1} by first establishing a characterisation of negative type diversities based on a metric on the power set. 

The next two sections return to the problem of embedding negative type diversities. In Section~\ref{sec:embed} we derive a geometric representation of a negative type diversity. Define 
 $(\Re^k,\delta_{neg})$ by 
\[\dn(A) = \sum_{i=1}^k \max\{a_i:a \in A\} - \min\left\{\sum_{i=1}^k a_i:a\in A\right\}.\]
The expression for $\dn$ bears substantial resemblance to that for an $L_1$-embeddable diversity. 
We prove in Theorem~\ref{thm:embedRk}  that a finite diversity $(X,\delta)$ is negative type if and only if it can be embedded in 
$(\Re^k,\delta_{neg})$ for some $k$.

With this result in hand, we look at the problem of embedding (finite) negative type diversities into $L_1$. There is a lower bound of $\Omega(\sqrt{\log n})$ for the metric case, however this does not directly imply a bound for negative type diversities as the induced metric of any negative type diversity is already $L_1$-embeddable. We follow a different strategy and use results on  Cheegar constants for hypergraphs to prove a  $\Omega(\sqrt{\log n})$ bound in the diversity case.

%
%
%
%
%
%
%
%
%
%

\section{Characterising negative type diversities} \label{sec:char}

In this section, we establish some basic characterisations and properties of negative type diversities on finite sets.


\begin{lem} \label{lem:phisquared}
For a finite set $X$ and real-valued function $\phi$ defined on $\pow(X)$, 
let $M$ be the matrix with rows and columns indexed by $\pow(X)$ and $M_{AB} = \phi(A \cup B)$ for all $A,B \in \pow(X)$.  Define the vector $\lambda$ by 
	\begin{equation}
	\lambda_A = \sum_{B:A \subseteq B} (-1)^{|A|+|B|+1} \phi(B). \label{def:lambdaPhi}
	\end{equation}
Then for all $\bx \in \Re^{\pow(X)}$ we have
\begin{equation}
	\sum_{A,B} \bx_A \bx_B \phi(A \cup B) =  - \sum_C \lambda_C  \Big( \sum_{A  \subseteq C}  \bx_A \Big)^2. \label{eq:xMx}
\end{equation}
\end{lem}
\begin{proof}	
Applying Moebius inversion to \eqref{def:lambdaPhi} we obtain
\[ \phi(A) = - \sum_{B : A \subseteq B} \lambda_B\]
for all $A \in \pow(X)$ (see  \cite{aigner2012combinatorial}). Hence for  $\bx \in \Re^{\pow(X)}$ we have 
\begin{align*}
\sum_{A,B} \bx_A \bx_B \phi(A \cup B) & = -  \sum_{A,B} \bx_A \bx_B \sum_{C : A \cup B \subseteq C}  \lambda_C \nonumber \\
& =  - \sum_C \lambda_C  \sum_{A,B: A \cup B \subseteq C}  \bx_A \bx_B \nonumber  \\ 
& = - \sum_C \lambda_C  \left( \sum_{A  \subseteq C}  \bx_A \right)^2. 
\end{align*}
\end{proof}

We  now prove the first  characterisation theorem for negative type diversities.

\begin{thm} \label{thm:negLambda} Let $X$ be a finite set and let $\delta$ be a real-valued function on $\pow(X)$ such that $\delta(A) = 0$ whenever $|A| \leq 1$. For all $A \in \pow(X)$ define
\begin{equation} \lambda_A = \sum_{B:A \subseteq B} (-1)^{|A| + |B| + 1} \delta(B). \label{def:lambda}\end{equation}
Then $(X,\delta)$ is a negative type diversity if and only if $\lambda_A \geq 0$ for all $A \neq \emptyset, X$. Furthermore,
\begin{equation}
\delta(A) = - \sum_{B: A \subseteq B} \lambda_B
\label{def:invlamba}
\end{equation}
for all $A\subseteq X$.
 \end{thm}
 \begin{proof}
 
 Suppose that $(X,\delta)$ is a negative type diversity. Fix $U \neq \emptyset,X$ and   define $\bx$ by
\[\bx_A  = \begin{cases} (-1)^{|A| + |U|} & \mbox{ if $U \subseteq A$ } \\ 0 & \mbox{ otherwise;} \end{cases}\]
Then, by Moebius inversion,
\[ \sum_{A:A  \subseteq C}  \bx_A  = \begin{cases} 1 & \mbox{ if $C = U$} \\ 0 & \mbox{ otherwise } \end{cases}\]
while
\[ \sum_{A} \bx_A = \sum_{A:U \subseteq A}  (-1)^{|A| + |U|}  = 0,\]
since $U \neq X$. Also, $\bx_\emptyset= 0$ because $U \neq \emptyset$.
Hence by the definition of negative type diversities and Lemma~\ref{lem:phisquared} we have
\[0 \geq \sum_{A,B} \bx_A \bx_B \delta(A \cup B) =  - \sum_C \lambda_C  \Big( \sum_{A  \subseteq C}  \bx_A \Big)^2 = -\lambda_U.\]

For the converse, suppose that $\lambda_A \geq 0$ for all $A \neq \emptyset,X$. Suppose $\bx$ is any vector in $\Re^{\pow(X)}$ with $\bx_\emptyset = 0$ and $\sum_A \bx_A = 0$. By Lemma~\ref{lem:phisquared} we have
\begin{align*}
 \sum_{A,B} \bx_A \bx_B \phi(A \cup B) &=  - \sum_C \lambda_C  \Big( \sum_{A  \subseteq C}  \bx_A \Big)^2 \\
 &=  - \lambda_\emptyset \bx_\emptyset^2  - \lambda_X  \Big( \sum_{A  \subseteq X}  \bx_A \Big)^2
  - \sum_{C \neq \emptyset,X}  \lambda_C  \Big( \sum_{A  \subseteq C}  \bx_A \Big)^2   \\
  & = - \sum_{C \neq \emptyset,X}  \lambda_C  \Big( \sum_{A  \subseteq C}  \bx_A \Big)^2   \\
  & \leq 0.
  \end{align*}
  
  Suppose $A \subseteq B$. Define the vector $\by$ by $\by_A = 1$, $\by_B = -1$ and $\by_C = 0$ for all $C \neq A,B$. Then 
\[0 \geq  \sum_{A,B} \bx_A \bx_B \phi(A \cup B) = \delta(A \cup A) - 2 \delta(A \cup B) + \delta(B \cup B) = \delta(A) - \delta(B).\]
It follows that $\delta$ is monotonic. 

  Now consider arbitrary $A,B \in \pow(X)$ and $c \in X$. Then
\begin{align*}
\delta(A \cup \{c\}) + \delta(B \cup \{c\}) - \delta(A \cup B \cup \{c\}) \hspace{-4cm} & \\ &= 	\delta(A \cup \{c\}) + \delta(B \cup \{c\}) - \delta(A \cup B \cup \{c\}) - \delta(\{c\}) \\
& = -\sum_{V:A \cup \{c\} \subseteq V} \lambda_V -\sum_{V:B \cup \{c\} \subseteq V} \lambda_V+ \sum_{V:A \cup B \cup \{c\} \subseteq V} \lambda_V + \sum_{V:c\in V} \lambda_V \\
& = \sum_{\substack{V:c \in V\\ A \not \subseteq V,\, B \not \subseteq V}} \lambda_V \\
& \geq 0.
\end{align*}
This, together with monotonicity, implies the triangle inequality for diversities. Hence $(X,\delta)$ is a diversity with negative type.
\end{proof}

A direct consequence of Theorem~\ref{thm:negLambda} is that the space of negative type diversities on a finite set $X$ forms a polyhedral cone. This cone has dimension $2^{|X|} - |X| - 1$, the same as the dimension of the cone of diversities on $X$.

We present several examples of negative type diversities.

\begin{prop}
\begin{enumerate}
\item Every diversity on three points is negative type.
\item There is a diversity on four points which is not negative type.
\item Every finite $L_1$-embeddable diversity is negative type, though there are negative type diversities which are not $L_1$-embeddable.
\item For finite $X$, if $(X,\delta)$ is the diversity with $\delta(A) = 1$ whenever $|A|>1$ then $(X,\delta)$ is negative.
\end{enumerate}
\end{prop}
\begin{proof}
\begin{enumerate}
\item Given a diversity $(\{a,b,c\},\delta)$ we have from \eqref{def:lambda} that
\begin{align*}
\lambda_{\{a\}} &= \delta(\{a,b\}) + \delta(\{a,c\}) - \delta(\{a,b,c\}) \geq 0 \\
\lambda_{\{a,b\}} &= \delta(\{a,b,c\}) - \delta(\{a,b\}) \geq 0.
\end{align*}
By symmetry, and Theorem~\ref{thm:negLambda}, $(\{a,b,c\},\delta)$ is negative.
\item Let $(\{a,b,c,d\},\delta)$ be the diversity with $\delta(A) = \left \lceil \frac{|A|}{2} \right\rceil$ for $A \subseteq \{a,b,c,d\}$ with $|A|>1$. Then
$\lambda_{\{a\}} = -1$ so $(\{a,b,c,d\},\delta)$ is not negative.
\item For any cut diversity $\delta = \delta_{C|\overline{C}}$ on $X$ we have $\lambda_X = -1$, 
$\lambda_A = 1$ if $A = C$ or $A = C$, and $\lambda_A = 0$ otherwise. By Theorem~\ref{thm:negLambda}, $\delta$ is negative, and since every $L_1$-embeddable diversity is a non-negative combination of cut diversities, so is every $L_1$-embeddable diversity. Any diversity on three points which does not satisfy $\delta(\{a,b,c\}) = (\delta(\{a,b\}) + \delta(\{a,c\})+\delta(\{b,c\}))/2$ is negative but not $L_1$-embeddable.
\item For any $A \neq X$ with $|A|>1$ we have 
\[\lambda_A = \sum_{B:A \subseteq B} (-1)^{|A|+|B|+1} \delta(B) = \sum_{B:A \subseteq B} (-1)^{|A|+|B|+1} = 0 \]
so by Theorem~\ref{thm:negLambda}, $(X,\delta)$ is negative.
\end{enumerate}
\end{proof}

Schoenberg's theorem states that every negative type metric is isometric to the square of a Euclidean metric. We do not have a direct analogue of this result for diversities, however the metric result leads to an appealing property of negative type diversities. 

In what follows let $\pownon(D)$ be the set of all non-empty subsets of $X$.

\begin{prop} \label{Prop:NegDivMetric}
Let $(X,\delta)$ be a finite diversity and let $D$ be the symmetric real-valued function defined on 
 $\pownon(X) \times \pownon(X)$ given by 
\[D(A,B) = \delta(A \cup B) - \frac{1}{2}\delta(A) - \frac{1}{2}\delta(B).\]
The following are equivalent
\begin{enumerate}
\item $(X,\delta)$ is a negative type diversity;
\item $(\pownon(X),D)$ is an $L_1$-embeddable metric;	
\item $(\pownon(X),D)$ is a metric of negative type;
\item $(\pownon(X),D)$ is isometric to the square of a Euclidean metric.
\end{enumerate}
\end{prop}
\begin{proof}
(1) $\Rightarrow$ (2). From \eqref{def:invlamba} in Theorem~\ref	{thm:negLambda} we have
\begin{align*}
2 D(A,B) & = \delta(A \cup B) - \delta(A) + \delta(A \cup B) - \delta(B) \\
& = \sum_{C:A \subseteq C, B \not \subseteq C} \lambda_C + 	\sum_{C:A \not \subseteq C, B  \subseteq C} \lambda_C\\
& = \sum_{C} \lambda_C d_{\pownon(C)|\overline{\pownon(C)}},
\end{align*}
Here $d_{\pownon(C)|\overline{\pownon(C)}}$ is the cut metric for the cut $\pownon(C)|(\pownon(X)\setminus\pownon(C))$. So $D$ is a non-negative linear combination of split metrics, and is therefore an $L_1$-embeddable metric.\\
(2) $\Rightarrow$ (3). Every $L_1$-embeddable metric is a negative type metric.\\
(3) $\Rightarrow$ (1). For all $\bx$ such that $\bone^T\bx = 0$ and $\bx_\emptyset = 0$ we have 
\begin{align*}
0 & \geq \sum_{A \neq \emptyset} \sum_{B \neq \emptyset} \bx_A \bx_B D(A,B) 
=\sum_{A} \sum_{B} \bx_A \bx_B D(A,B)
\\
& = \sum_{A} \sum_{B} \bx_A \bx_B \delta(A \cup B) - \frac{1}{2}\sum_{A} \bx_A \delta(A) \sum_{B} \bx_B - \frac{1}{2} \sum_{B} \bx_B \delta(B) \sum_{A} \bx_A \\
& 
= \bx^T M \bx
\end{align*} 
where $M_{AB} = \delta(A \cup B)$ for all $A,B \in \pow(X)$. Hence $\delta$ is of negative type.\\
(3) $\Leftrightarrow$ (4) Schoenberg's theorem  \cite{schoenberg1935remarks,schoenberg1938metric}.
\end{proof}

If we look at $D$ restricted to singletons we obtain a surprising result.
\begin{thm} \label{thm:inducedL1}
The induced metric of a negative type diversity is $L_1$-embeddable.
\end{thm}
\begin{proof}
Suppose $(X,\delta)$ is negative. When we restrict $D$ to singleton sets, we see
\[D(\{a\},\{b\}) = \delta(\{a,b\}) - \frac{1}{2} \delta(\{a\}) - \frac{1}{2} \delta(\{b\}) = \delta(\{a,b\}).\]
Hence the induced metric of $(X,\delta)$ is isometric to $D$ restricted to singletons and, by Proposition~\ref{Prop:NegDivMetric}, $D$ is $L_1$-embeddable.	
\end{proof}

The relationship between $L_1$-embeddable metrics and $L_1$-embeddable diversities is straight-forward: the induced metric of any $L_1$-embeddable diversity is an $L_1$-embeddable metric and, conversely, every $L_1$-embeddable metric is the induced metric of some $L_1$-embeddable diversity.

The situation for negative type diversities is a bit more nuanced. The induced metric of any negative type diversity is an $L_1$-embeddable metric, and hence a negative type metric. But metrics which are negative but {\em not} $L_1$-embeddable are not the induced metrics of any negative type diversity.
  
 Furthermore, as there are metrics which require distortion $\Omega(\log n)$ to embed into $L_1$, there are diversities which will require distortion $\Omega(\log n)$  to embed into a negative type diversity. In contrast, any metric can be embedded in a negative type metric with only $O(\sqrt{\log n} \log \log n)$ distortion.

\section{A universal embedding for negative type diversities} \label{sec:embed}

We provide an example of a diversity  is that is {\em universal} for negative type diversities, in the sense that every (finite) negative type diversity can be embedded into this diversity, and every finite subset of this diversity induces a negative type diversity.
Define $(\Re^k,\delta_{\mathrm{neg}})$ by 
\[\dn(A) = \sum_{i=1}^k \max\{a_i: \ba \in A\} - \min\left\{\sum_{i=1}^k a_i:\ba\in A\right\}.\]
First we show that we can restrict our attention to embeddings where all vectors have zero sum.

\begin{lem} \label{lem:zeroNeg}
There is an embedding $\phi$ of $(\Re^k,\dn)$ into $(\Re^{k+1},\dn)$ such that $\bone^T \phi(\bx) = 0$ for all $\bx \in \Re^k$. 
\end{lem}
\begin{proof}
The map 
\[\phi:\Re^k \rightarrow \Re^{k+1}:\bx \mapsto (\bx,-\bone^T\bx)\]
satisfies the condition that $\bone^T \phi(\bx) = 0$ for all $\bx$. For all finite $A \subseteq \Re^k$ we have
\begin{align*}
\dn(\phi(A)) & = \sum_{i=1}^k \max\{a_i:\ba \in A\} + \max\{-\bone^T\ba:\ba \in A\} - 0 \\
& = \dn(A).
\end{align*}
\end{proof} 
 
 \begin{thm} \label{thm:embedRk}
 A finite diversity $(X,\delta)$ is of negative type if and only if can be embedded in $(\Re^k,\dn)$ for some $k$.
 \end{thm}
 \begin{proof}
 First we show that if $\sX$ is any finite subset of $\Re^k$ then $(\sX,\dn)$ is of negative type. By Lemma~\ref{lem:zeroNeg} we can assume, without loss of generality, that $\bone^T \bx = 0$ for all $\bx \in \sX$. 
 
 For each $A \subseteq \sX$, we have
\begin{align*}
\lambda_A & = \sum_{B:A \subseteq B \subseteq \sX} (-1)^{|A|+|B|+1} \delta(B)\\ 
& = \sum_{i=1}^k \sum_{B:A \subseteq B \subseteq \sX} (-1)^{|A|+|B|+1} \max\{b_i:b \in B\}.
\end{align*}
By Theorem~\ref{thm:negLambda} we need to show that $\lambda_A \geq 0$ for $A \neq \emptyset, \sX$.
Fix $i$, and suppose that that $\sX$ is ordered as $\sX = \{\bx_1,\ldots,\bx_n\}$ such that 
\[x_{1i} \leq x_{2i} \leq \cdots \leq x_{ni}.\]
For $B \subseteq \sX$ define $m(B) = \max\{j:\bx_{j} \in B\}$. Then
\begin{align*}
\sum_{B:A \subseteq B \subseteq \sX} (-1)^{|A|+|B|+1} \max\{b_i:b \in B\} & = \sum_{j=1}^n x_{ji} \sum_{\substack{B:A \subseteq B \subseteq \sX\\m(B) = j}} (-1)^{|A|+|B|+1} \\
\intertext{
If there is $\bx_\ell \in \sX \setminus A$ such that $\ell < m(A)$, or if $m(A) <  j-1$ then}
\sum_{\substack{B:A \subseteq B \subseteq \sX\\m(B) = j}} (-1)^{|A|+|B|+1} &= 0 \\ \intertext{otherwise}
\sum_{j=1}^n x_{ji} \sum_{\substack{B:A \subseteq B \subseteq \sX\\m(B) = j}} (-1)^{|A|+|B|+1} &= x_{(m(A)+1)i} - x_{m(A)i}  
 \geq 0.
\end{align*}
Hence $\lambda_A \geq 0$ for all $A \subset \sX$ such that $A \neq \sX$, and $(\sX,\dn)$ is of negative type.

For the converse, suppose $|X| = n$ and $(X,\delta)$ is of negative type. Let $\lambda$ be given by \eqref{def:lambda}, let $k = 2^{n}-1$ and suppose that the dimensions of $\Re^k$ are indexed by nonempty subsets of $X$. Let $\phi:X \rightarrow \Re^k$ be the map with
\begin{align*}
\phi(x)_B & = \begin{cases} - \lambda_B & x \in B \\
0 & \mbox{otherwise,} \end{cases}\\ 
\intertext{for all $x \in X$ and nonempty $B \subseteq X$. Then}
\max\{\phi(a)_B:a \in A\} &= \begin{cases} -\lambda_B & A \subseteq B \\
0  & \mbox{otherwise.} \end{cases}\\ \intertext{For all $x \in X$,}
\sum_B \phi(x)_B &= -\sum_{B:x \in B \subseteq X} \lambda_B = \delta(\{x\})= 0\\ \intertext{
and for all $A \subseteq X$ we have}
\dn(\phi(A))  = -\sum_{B:A \subseteq B} \lambda_A = \delta(A),
\end{align*}
where we have used expression \eqref{def:invlamba} in Theorem~\ref{thm:negLambda}.
\end{proof}

\section{$L_1$ embedding} \label{sec:L1}

Much of the recent interest in negative type metrics relates to embedding into $L_1$ \cite{naor17Integrality,Lee2005a,Krauthgamer2005,krauthgamer2009improved,arora2009expander,khot2015unique,Lee2005}. Every $n$-point negative type metric can be embedded into $L_1$ with distortion at most $O(\sqrt{\log n}\log\log n)$, with a lower bound of 
$\Omega(\sqrt{{\log n}})$ \cite{arora2008euclidean}. These results lead to the  current best approximation bound for sparsest cut and other problems.

Here we investigate the problem of embedding negative type diversities into $L_1$. The aim is to investigate whether the approximation algorithms based on embeddings of negative type metrics can be extended to algorithms based on embedding negative type diversities. We have already shown that algorithms of \cite{linial1995geometry} and others based on embedding metrics in $L_1$ have direct analogies for diversities \cite{BryantTupper2014}. 

Every $L_1$-embeddable diversity is negative; we start by characterising which negative type diversities are $L_1$-embeddable.

\begin{prop}
Let $(X,\delta)$ be a negative type diversity and let $\lambda$ be given by \eqref{def:lambda}. Then $(X,\delta)$ is $L_1$-embeddable if and only if $\lambda_A = \lambda_{\overline{A}}$ for all $A \neq \emptyset,X$.
\end{prop}
\begin{proof}
Suppose $\delta = \delta_{C|\overline{C}}$. First, suppose $\delta_{C|\overline{C}}(A)=1$. For any $B$ with $A \subseteq B$ we have $\delta_{C|\overline{C}}(B)=1$ too.  Then
\begin{align*}
\lambda_A & = \sum_{B:A \subseteq B} (-1)^{|A|+|B|+1} \delta_{C|\overline{C}}(B) \\
&= \sum_{B:A \subseteq B} (-1)^{|A|+|B|+1} 
\end{align*}
which is $0$ unless $A=X$, in which case it is $-1$.
Now suppose that $A \subseteq C$. We write $\lambda_A$ using $B= A \cup D$ as
\begin{align*}
\lambda_A & = \sum_{D: D \cap A = \emptyset} (-1)^{2|A|+|D|+1} \delta_{C|\overline{C}}(A \cup D) 
\\
&= \sum_{D: D \cap A = \emptyset, D \cap \overline{C} \neq \emptyset} (-1)^{|D|+1}  
\end{align*}
which is $0$ unless $A=C$, in which case it is $-1$. Similarly for $A \subseteq \overline{C}$, $\delta_A = -1$ if $A=\overline{C}$ and is $0$ otherwise. Summarizing, for cut diversities $\delta_{C|\overline{C}}$, $\lambda_X=\lambda_C=\lambda_{\overline{C}}=-1$, and all other $\lambda_A =0$. 
So, if $\delta$ is a cut diversity then $\lambda_A = \lambda_{\overline{A}}$ for all $A \neq \emptyset, X$. This will hold also for any $L_1$ diversity, since $L_1$ diversities are non-negative linear combinations of cut diversities.

For the converse, suppose that the diversity $\delta$ of negative type is such that $\lambda_A = \lambda_{\overline{A}}$ for all $A$. Let $\tilde{\delta}$ be the $L_1$-embeddable diversity 
\[\tilde{\delta} = \frac{1}{2} \sum_{C:C \subseteq X} \lambda_C \delta_{C|\overline{C}}\]
and with $\tilde{\lambda}$  given by
\[\tilde{\lambda}_A = \sum_{B:A \subseteq B} (-1)^{|A|+|B|+1} \tilde{\delta}(B).\]
Then, when $A \neq \emptyset, X$,
\begin{eqnarray*}
\tilde{\lambda}_A & = & \sum_{C} \frac{1}{2} \lambda_C \sum_{B : A \subseteq B} (-1)^{|A|+|B|+1} 
\delta_{C | \overline{C}} (B) \\
& = &  (\lambda_A + \lambda_{\overline{A}})/2 = \lambda_A.
\end{eqnarray*}
So $\delta$ and $\tilde{\delta}$ have the same $\lambda$ vector.
As the map from $\delta$ to $\lambda$ is invertible, $\tilde{\delta} = \delta$ and $\delta$ is $L_1$-embeddable.
\end{proof}

There are negative type metrics on $n$ points which cannot be embedded in $L_1$ with distortion less than $\Omega(\sqrt{\log n})$. However, as the induced metric of any negative type diversity is $L_1$-embeddable the general bound for metrics does not imply a bound for diversities. We instead follow a different strategy to show that there are negative type diversities which still require at least $\Omega(\sqrt{\log n})$ distortion to embed the diversity into $L_1$. Our bound is based on connections in \cite{BryantTupper2014} between $L_1$ embedding of diversities and sparsest cut problems in hypergraphs.

Given $m>2$ let $n = \binom{2m}{m}$ and let $\sX$ be the set $\binom{[2m]}{m}$ of all subsets of $[2m] = \{1,2,\ldots,2m\}$ of cardinality $m$. Let $(\sX,\dh)$ be the diversity with 
\[\delta(\sA) = \left| \bigcup_{A \in \sA} A \right| - m\]
for all non-empty $\sA \subseteq \sX$.

\begin{prop}
$(\sX,\dh)$ is of negative type.
\end{prop}
\begin{proof}
A diversity isomorphism from $(\sX,\dh)$ to $(\Re^n,\dn)$ is given by mapping each set in $\sX$ to the corresponding vector of $m$ ones and $n-m$ zeros. The diversity $(\sX,\dh)$ is then of negative type by Theorem~\ref{thm:embedRk}.
\end{proof}

Our lower bound for embedding $(\sX,\dh)$ into $L_1$ is based on a lower bound for hypergraph cuts.
Suppose that $\sU$ is a non-empty subset of $\sX$. Define
\begin{align*}
\Int(\sU) & = \left\{B \in \binom{[2m]}{m+1}:\binom{B}{m} \subseteq \sU \right\} \\
\Ext(\sU) & = \left\{B \in \binom{[2m]}{m+1}:\binom{B}{m} \subseteq \overline{\sU} \right\} \\
\partial \sU & = \binom{[2m]}{m+1} \setminus (\Int(\sU) \cup \Ext(\sU)).
\end{align*}
We derive a lower bound for $\frac{|\partial \sU|}{|\sU||\overline{\sU}|}$. Without loss of generality we assume $|\Int(\sU)| \leq |\Ext(\sU)|$, swapping $\sU$ for $\overline{\sU}$ if this is not the case. Hence
\begin{align}
 |\Int(\sU)|  \leq \frac{1}{2} \binom{2m}{m+1} \label{eq:IntUupper}.
 \end{align}

\begin{lem} \label{Hypercut}
\[\frac{|\partial \sU|}{|\sU|} \geq  \frac{m}{m+1} \left(1+5 \sqrt{2m}\right)^{-1}.\] 
\end{lem}
\begin{proof}
Let $\sJ = (\sY,\sE)$ be the graph with vertex set $\sY = \binom{[2m]}{m+1}$ and edge set 
\[\sE = \{\{A,B\}:|A \cap B| = m\}.\]
 The graph  $\sJ$ is a Johnson graph \cite{brouwer2012distance}.  For any subset $\sY' \subseteq \sY$ we define the vertex boundary $b(\sY')$ by
\[b(\sY') = \{B \in \sY: B \not \in \sY' \mbox{ but there is $A \in \sY'$ such that $\{A,B\} \in \sE$} \}.\]
Theorem 2 of \cite{Christofides2013}  provides a lower bound on the size of $b(\sY')$. Let $\alpha = \frac{|\sY'|}{|\sY|}$, then
\begin{align}
|b(\sY')| &\geq \frac{1}{5} \sqrt{ \frac{ 2m }{(m+1)(m-1)} } \alpha (1-\alpha) \binom{2m}{m+1} \nonumber \\
& \geq \frac{1}{5} \sqrt{ \frac{ 2 }{m} } \binom{2m}{m+1} \alpha (1-\alpha)  \label{eq:Christofides}
\end{align}

Let $\sY' = \Int(\sU)$. Then $B \in b(\Int(\sU))$ implies that $B \not \in \Int(\sU)$ but that there is $A \in \Int(\sU)$ such that $\{A,B\} \in \sE$. Let $C = A \cap B$. As $A \in \Int(\sU)$ we have $C \in \binom{A}{m} \subseteq \sU$. Since $C \in \binom{B}{m}$ we have $\binom{B}{m} \not \subseteq \overline{\sU}$ and so $B \not \in \Ext(\sU)$. The only remaining possibility is that  $B \in \partial \sU$ and hence 
\begin{equation}
b(\Int(\sU)) \subseteq \partial \sU. \label{eq:bSubset}
\end{equation} 

Since $\sY' = \Int(\sU)$, we have that $\alpha = \frac{|\Int(\sU)|}{|\sY|}$. Then from 
\eqref{eq:IntUupper} we have $\alpha \leq 0.5$ and from  \eqref{eq:Christofides} and \eqref{eq:bSubset}  we have
\begin{align*}
|\partial \sU| & \geq |b(\Int(\sU))| \\
& \geq \frac{1}{5} \sqrt{ \frac{ 2 }{m} } \alpha (1-\alpha) \binom{2m}{m+1} \\
& \geq \frac{1}{10}   \sqrt{ \frac{ 2 }{m} } |\Int(\sU)|.
\end{align*}
Reversing the inequality we now obtain
\begin{align*}
|\Ext(\sU)| & = |\sY| - |\Int(\sU)| - |\partial \sU| \\
& \geq |\sY|  - \left(\frac{10}{\sqrt{2}} \sqrt{m} \right) |\partial \sU| - |\partial \sU|.
\end{align*}
 Thus
 \begin{align}
 |\overline{\sU}| & \geq \left| \bigcup_{B \in \Ext(\sU)} \binom{B}{m} \right| \nonumber \\
 & \geq \frac{m+1}{m} |\Ext(\sU)| \label{mstep} \\
 & \geq \frac{m+1}{m} \left(|\sY| -\left(5 \sqrt{2m} + 1\right) |\partial \sU| \right). \nonumber
 \end{align}
 Inequality \eqref{mstep} holds since there are $\left| \binom{B}{m} \right| = m+1$ elements of $\overline{\sU}$ contained in each $B \in \Ext(\sU)$, and each of these could be contained in at most $m$ other sets $B' \in \Ext(\sU)$. A bound for $|\sU|$ follows immediately:
 \begin{align*}
 |\sU| & = |\sX| - |\overline{\sU}| \\
 & \leq \binom{2m}{m} - \frac{m+1}{m} \left(\binom{2m}{m+1} -\left(5 \sqrt{2m} + 1\right) |\partial \sU|\right) 
  \\
 & \leq \frac{m+1}{m} \left(1+5 \sqrt{2m}\right) |\partial \sU|
 \end{align*}
 \end{proof}
 
 The link between sparsest cuts of hypergraph and diversity embeddings was established by \cite{BryantTupper2014}. We make use of the same ideas. 
 
 \begin{thm} \label{thm:L1bound}
Let $k_1:=k_1(\dh)$ be the minimum distortion required to embed $(\sX,\dh)$ into $L_1$. 
Then 
\[k_1 \geq \frac{m}{4+20\sqrt{2m}} \in \Omega(\sqrt{m}).
\]
\end{thm}
\begin{proof}
Define vectors $\bc$ and $\bd$, indexed by $\pow(\sX)$, by
\begin{align*}
\bc_{\sA} &= \begin{cases} 1 & \sA = \binom{B}{m} \mbox{ for some $B \in \binom{[2m]}{m+1}$} \\
                   0 & \mbox{otherwise.} \end{cases}\\
\bd_{\sA} &= \begin{cases} 1 & |\sA| = 2 \\ 0 & \mbox{otherwise.} \end{cases} \\
\end{align*}
Then

\begin{align*}
\sum_{\sA \subseteq \sX} \bc_{\sA} \dh(\sA) & = \sum_{B \in \binom{[2m]}{m+1}} \dh \left(\binom{B}{m} \right) \\
& = \binom{2m}{m+1},
\end{align*}
since $\delta_H( \binom{B}{m}) = 1$ for all $B$.

Also
\begin{align*}
\sum_{\sA \subseteq \sX} \bd_{\sA} \dh(\sA) & = \frac{1}{2} \sum_{A,B \in \sX} \delta_H(\{ A, B\}) \\
& = \frac{1}{2} \sum_{A, B \in \sX} |\{ i \in [2m] \colon i \in A \setminus B\}| \\
& = \frac{1}{2} \sum_{i=1}^{2m} |\{(A,B) \in \sX \times \sX: i \in A \setminus B \} | \\
& = \frac{1}{2}(2m) \binom{2m-1}{m-1} \binom{2m-1}{m}\\
& = m \binom{2m-1}{m}^2.
\end{align*}
Let $\delta_1$ be any $L_1$-embeddable diversity such that for all $\sA \subseteq \sX$,
\[\dh(\sA) \leq \delta_1(\sA) \leq k_1 \dh(\sA).\]
We then have
\begin{align*}
\sum_{\sA \subseteq \sX} \bc_{\sA} \delta_1(\sA) & \geq  \sum_{\sA \subseteq \sX} \bc_{\sA} \dh(\sA) \\
\sum_{\sA \subseteq \sX} \bd_{\sA} \delta_1(\sA) & \leq k_1 \sum_{\sA \subseteq \sX} \bd_{\sA} \dh(\sA) \\.
\end{align*}

which implies:
\[k_1\frac{\sum_{\sA \subseteq \sX} \bc_{\sA} \dh(\sA)}{\sum_{\sA \subseteq \sX} \bd_{\sA} \dh(\sA)}\ge \frac{\sum_{\sA \subseteq \sX} \bc_{\sA} \delta_1(\sA)}{\sum_{\sA \subseteq \sX} \bd_{\sA} \delta_1(\sA)}\]

From the proof of Proposition 16 in \cite{BryantTupper2014} there is a cut metric  $\delta_{\sU|\overline{\sU}}$ such that 
\begin{align*}
\frac{\sum_{\sA \subseteq \sX} \bc_{\sA} \delta_{\sU|\overline{\sU}}(\sA)}{\sum_{\sA \subseteq \sX} \bd_{\sA} \delta_{\sU|\overline{\sU}}(\sA)} &\leq \frac{\sum_{\sA \subseteq \sX} \bc_{\sA} \delta_1(\sA)}{\sum_{\sA \subseteq \sX} \bd_{\sA} \delta_1(\sA)}.
\end{align*}
Now note that
\begin{align*}
\frac{\sum_{\sA \subseteq \sX} \bc_{\sA} \delta_{\sU|\overline{\sU}}(\sA)}{\sum_{\sA \subseteq \sX} \bd_{\sA} \delta_{\sU|\overline{\sU}}(\sA)} & = \frac{|\partial \sU|}{|\sU| |\overline{\sU}|}\\  
& \geq \frac{m}{m+1}(1+5\sqrt{2m})^{-1}|\overline{\sU}|^{-1}\\
& \geq \frac{m}{m+1}(1+5\sqrt{2m})^{-1}\binom{2m}{m}^{-1}\\
\end{align*}

On the other hand:
\begin{align*}
\frac{\sum_{\sA \subseteq \sX} \bc_{\sA} \dh(\sA)}{\sum_{\sA \subseteq \sX} \bd_{\sA} \dh(\sA)}&=\frac{\binom{2m}{m+1}}{m\binom{2m-1}{m}^2}\\
&=\frac{2}{(m+1)\binom{2m-1}{m}}
\end{align*}

Bringing everything together:
\begin{align*}
k_1&\ge\left(\frac{m}{m+1}(1+5\sqrt{2m})^{-1}\binom{2m}{m}^{-1}\right)/\left(\frac{2}{(m+1)\binom{2m-1}{m}}\right)\\
&=\frac{m}{4+20\sqrt{2m}}.
\end{align*}
\end{proof}

\begin{cor} \label{cor:lowerbound}
For every $m$ there is a negative type diversity with $n = \binom{2m}{m}$ points which cannot be embedded in $L_1$ with less than $\Omega(\sqrt{m}) = \Omega(\sqrt{\log n})$ distortion.	
\end{cor}

\section{Open problems}

Corollary~\ref{cor:lowerbound} gives a lower bound for the largest distortion required to embed a negative diversity into $L_1$. The current best upper bound for embedding a negative type metric into $L_1$ is $O(\sqrt{\log n}\log\log n)$, as established by  \cite{arora2008euclidean}. The ideas used in the embedding do not appear to generalise to diversities in a straight-forward way, and determining the best upper bound for embedding a negative diversity into $L_1$ is an open problem.

A related question is to determine the distortion bounds for embedding an arbitrary finite diversity into a negative type diversity. Since the induced metric of any negative diversity is itself $L_1$ embeddable, there will be diversities which cannot be embedded into a negative diversity with distortion less the $O(\log n)$. It is an open problem as to what the matching upper bound might be.

In a slightly different direction, we note that the definition of negative type diversities could be extended to a class of diversities with $p$-negative type: those which satisfy 
\[\sum_{A} \sum_{B} \bx_A \bx_B \delta(A \cup B)^p \leq 0\]
for all zero sum vectors $\bx \in \Re^{\pow(X)}$ with $\bx_\emptyset = 0$. It would be of interest to see whether results on metrics with $p$-negative type have analogues for diversities \cite{Li10,Sanchez15,Wolf12}.

\bibliographystyle{apa}
\bibliography{NegativeDiversities}

\end{document}